\documentclass[12pt,a4paper]{scrartcl}

\usepackage[utf8]{inputenc}
\usepackage[T1]{fontenc}
\usepackage[english]{babel}

\usepackage{graphicx}
\usepackage{latexsym}
\usepackage{amsmath,amssymb,amsthm}

\usepackage{enumitem}

\usepackage{lineno}

\usepackage{bbm}

\usepackage{mathtools}
\usepackage{abstract}

\theoremstyle{definition}

\newtheorem{definition}{Definition}[section]

\newtheorem{theorem}[definition]{Theorem}
\newtheorem{lemma}[definition]{Lemma}       
\newtheorem{question}[definition]{Question}          

\numberwithin{equation}{section} 

\title{Cardinal characteristics on $\kappa$ \\ modulo non-stationary}
\author{Johannes Philipp Sch\"urz \footnote{supported by FWF project I3081}}
\date{}

\begin{document}

\maketitle

\begin{abstract}
For $\kappa$ regular and uncountable we define variants of the classical cardinal characteristics modulo the non-stationary ideal.
\end{abstract} 

\section{Introduction}
Cardinal characteristics of $\mathfrak{P}(\kappa)$ for $\kappa$ at least inaccessible have been studied extensively in the last few years: \cite{Montoya}, \cite{Fischer-Soukup2}, \cite{Fischer-Soukup1}, \cite{Raghavan-Shelah1} and \cite{Raghavan-Shelah2} are just a few examples. Similar to the classical case on $\omega$, these `higher' cardinal characteristics are usually defined modulo the bounded ideal, e.g. $x$ is almost disjoint to $y$ iff $\vert x \cap y \vert < \kappa$ for $x,y \subseteq \kappa$ and $\mathfrak{a}_\kappa:= \min \{\vert \mathcal{A} \vert \colon \mathcal{A} \subseteq \mathfrak{P}(\kappa) \,\, \text{is a maximal almost disjoint family} \land \vert\mathcal{A} \vert \geq \kappa\}$. The cardinal characteristics  $\mathfrak{p}_\kappa, \mathfrak{t}_\kappa, \mathfrak{s}_\kappa, \mathfrak{b}_\kappa, \mathfrak{d}_\kappa, \mathfrak{r}_\kappa$ and $\mathfrak{u}_\kappa$ are defined similarly.\\
Let $\kappa$ be regular uncountable. In this paper we intend to define variants of these `higher' cardinal characteristics modulo the non-stationary ideal. To this end we define the club filter $Cl:=\{ x \subseteq \kappa \colon \exists y \subseteq x \,\, y \,\, \text{is club}\}$, the non-stationary ideal $NS:=\{ x \subseteq \kappa \colon \exists y \in Cl \,\, x \cap y = \emptyset \}$ and the set of stationary sets $St:= \mathfrak{P}(\kappa) \setminus NS$. Note that while the property $x \in Cl$ is upwards absolute for models with the same cofinalities, the properties $x \in NS$ and $x \in St$ are in general not.\\
\\
We will now define several relations on $St \times St$ modulo the non-stationary ideal and use them to define cardinal characteristics of $St$:

\begin{definition}
Let $x,y \in St$. We define:
\begin{itemize}
\item $y$ stationarily splits $x$ iff $x \cap y \in St$ and $x \setminus y\in St$.\\
$\mathfrak{s}_\kappa^{cl}:= \min \{\vert \mathcal{S}\vert \colon \mathcal{S} \subseteq St \land \forall x \in St \,\, \exists y \in \mathcal{S} \,\, y \,\, \text{stationarily splits}\,\,x\}$ the stationary splitting number and\\
$\mathfrak{r}_\kappa^{cl}:=\min \{\vert \mathcal{R} \vert \colon  \mathcal{R} \subseteq St \land \forall x \in St \,\, \exists y \in \mathcal{R} \,\, \neg(x \,\, \text{stationarily splits}\,\,y)\}$ the stationary reaping number
\item $x \subseteq_{cl}^* y$ iff $x \setminus y \in NS$.\\
$\mathcal{F} \subseteq St$ has the ${<}\kappa$-stationary intersection property iff for every $\mathcal{F}' \subseteq \mathcal{F}$ of size ${<}\kappa$ we have that $\bigcap_{x \in \mathcal{F}'} x \in St$.\\
$\mathfrak{p}_\kappa^{cl}:=\min \{ \vert \mathcal{P} \vert \colon \mathcal{P} \subseteq St \land \mathcal{P} \,\, \text{has the}\,\, {<}\kappa\text{-stationary intersection property}\,\, \land \neg( \exists x \in St \,\, \forall y \in \mathcal{P}\,\, x \subseteq_{cl}^* y)\}$ the stationary pseudo intersection number\\
$\mathfrak{t}_\kappa^{cl}:=\min \{\vert \mathcal{T} \vert \colon \mathcal{T} \subseteq St \land \mathcal{T}  \,\, \text{has the}\,\, {<}\kappa\text{-stationary intersection property}\,\, \land \mathcal{T} \,\,\text{is well-}$ $ \text{ordered by} \,\, \prescript{*}{cl}{\supseteq} \land \neg( \exists x \in St \,\, \forall y \in \mathcal{T}\,\, x \subseteq_{cl}^* y)\}$ the stationary tower number \footnote{Note that the notions of $\mathfrak{p}_\kappa^{cl}$ and $\mathfrak{t}_\kappa^{cl}$ introduced here are different to the ones defined in \cite{Fischer-Soukup2}.}
\item $x$ is stationary almost disjoint to $y$ iff $x \cap y \in NS$.\\
$\mathfrak{a}_\kappa^{cl}:= \min \{\vert \mathcal{A} \vert \colon \mathcal{A} \,\, \text{is a maximal stationary almost disjoint family} \land \vert \mathcal{A}\vert \geq \kappa\}$ the stationary almost disjointness number
\item $\mathfrak{u}_\kappa^{cl}:=\min \{ \vert \mathcal{B} \vert \colon \mathcal{B} \subseteq St \land \exists \mathcal{U} \subseteq \mathfrak{P}(\kappa) \,\,\mathcal{U}\,\,\text{is an ultrafilter}\,\,  \land \mathcal{B} \,\, \text{is a base for}\,\, \mathcal{U}\}$ \footnote{Hence $Cl \subseteq \mathcal{U}$.} the stationary ultrafilter number\\
$\mathfrak{u}_\kappa^{cl^*}:=\min \{ \vert \mathcal{B} \vert \colon \mathcal{B} \subseteq St \land \exists \mathcal{U} \subseteq \mathfrak{P}(\kappa) \,\,\mathcal{U}\,\,\text{is an ultrafilter}\,\,  \land \mathcal{B} \cup Cl \,\, \text{is a subbase for}\,\, \mathcal{U} \}\,$\footnote{i.e. $\{ y \in St \colon \exists x \in \mathcal{B}\, \exists cl \in Cl \,\, y= x \cap cl\}$ is a base for $\mathcal{U}$, since w.l.o.g. $\mathcal{B}$ is closed under intersections.} the stationary$^*$ ultrafilter number\\
$\mathfrak{u}_\kappa^{me}:=\min \{ \vert \mathcal{B} \vert \colon \mathcal{B} \subseteq St \land \exists \mathcal{U} \subseteq \mathfrak{P}(\kappa) \,\,\mathcal{U}\,\,\text{is a measure}\,\,  \land \mathcal{B}  \,\, \text{is a base for}\,\, \mathcal{U}\}$ \footnote{i.e. $\mathcal{U}$ is a ${<}\kappa$-complete ultrafilter.} the measure ultrafilter number\\
$\mathfrak{u}_\kappa^{nm}:=\min \{ \vert \mathcal{B} \vert \colon \mathcal{B} \subseteq St \land \exists \mathcal{U} \subseteq \mathfrak{P}(\kappa) \,\,\mathcal{U}\,\,\text{is a normal measure}\,\,  \land \mathcal{B}  \,\, \text{is a base for}\,\, \mathcal{U}\}$ the normal measure ultrafilter number\\
$\mathfrak{u}_\kappa^{nm^*}:=\min \{ \vert \mathcal{B} \vert \colon \mathcal{B} \subseteq St \land \exists \mathcal{U} \subseteq \mathfrak{P}(\kappa) \,\,\mathcal{U}\,\,\text{is a normal measure}\,\,  \land \mathcal{B} \cup Cl \,\, \text{is a subbase for}\,\, \mathcal{U}\}$ the normal measure$^*$ ultrafilter number
\item Let $f,g \in \kappa^\kappa$ and define $f \leq_{cl}^* g$ iff $\{ \alpha < \kappa \colon g(\alpha) < f(\alpha)\} \in NS$.\\
$\mathfrak{b}_\kappa^{cl}:=\min\{ \vert B\vert \colon B \subseteq \kappa^\kappa \land \forall g \in \kappa^\kappa \,\, \exists f \in B \,\, f \nleq_{cl}^* g\}$ the club unbounded number\\
$\mathfrak{d}_\kappa^{cl}:=\min \{\vert D \vert \colon D \subseteq \kappa^\kappa \land \forall f \in \kappa^\kappa \,\, \exists g \in D \,\, f \leq_{cl}^* g\}$ the club dominating number

\end{itemize}
\end{definition}

We will aim to establish some relations between these cardinal characteristics and also show some consistency results.\\

\section{Results / Questions}

The notions of club unbounded and dominating number have already been investigated by Cummings and Shelah (see \cite{Cummings}). In particular they showed the following theorem:

\begin{theorem}
Let $\kappa$ be regular uncountable. Then $\mathfrak{b}_\kappa= \mathfrak{b}_\kappa^{cl}$. If $\kappa \geq \beth_\omega$ then $\mathfrak{d}_\kappa=\mathfrak{d}_\kappa^{cl}$.
\end{theorem}

The stationary almost disjointness number $\mathfrak{a}_\kappa^{cl}$ is trivial:

\begin{lemma} \label{3L4}
Let $\kappa$ be regular uncountable. Then $\mathfrak{a}_\kappa^{cl}= \kappa$.
\end{lemma}

\begin{proof}
Partition $\kappa$ into $\kappa$many stationary sets $(x_i)_{i < \kappa}$. Define $y_i:= \kappa \setminus \bigcup_{j\leq i} x_i$ and set $x_\kappa:= \triangle_{i < \kappa}\, y_i$. Note that $x_i \cap x_\kappa \in NS$ for every $i< \kappa$. Now we have to distinguish two cases:
\begin{itemize}
\item If $x_\kappa \in St$, then we claim that the family $(x_i)_{i \leq \kappa}$ is maximal stationary almost disjoint. Towards a contradiction assume that $x^*\in St$ is stationary almost disjoint to $x_i$ for every $i \leq \kappa$. We define a function $f \colon x^* \rightarrow \kappa$ such that $f(k)$ is the unique $i< \kappa$ such that $k \in x_i$. Equivalently $f(k):= \min \{i < \kappa  \colon k \notin y_i\}$. If the set $\{k \in x^* \colon f(k) < k\}$ is stationary, then by Fodor's lemma (see \cite{Jech}) the set $\{k \in x^* \colon f(k)= \delta\}$ is stationary for some $\delta < \kappa$. But this implies that $x^* \cap x_\delta \in St$. Hence the set $\{k \in x^* \colon f(k) \geq k\}$ is stationary, and therefore $x^* \cap x_\kappa \in St$. But this also leads to a contradiction, hence $(x_i)_{i\leq \kappa}$ is a maximal stationary almost disjoint family.
\item If $x_\kappa \in NS$, then we proceed similarly and claim that $(x_i)_{i< \kappa}$ is maximal stationary almost disjoint. We define $f \colon x^* \rightarrow \kappa$ as above, and note that $\{k \in x^* \colon f(k) \geq k\}$ cannot be stationary. Hence there exists $\delta < \kappa$ such that $x^* \cap x_\delta \in St$.
\end{itemize}
\end{proof}

Let us say a few words about the spectrum of stationary almost disjointness:

\begin{definition}
We define $\text{Spec}_{\text{sad}}:=\{\gamma \geq \kappa \colon \exists \mathcal{A} \,\, \mathcal{A}\,\,\text{is a maximal stationary almost}$ $\text{disjoint family} \land  \vert \mathcal{A}\vert= \gamma\}$.
\end{definition}

\begin{definition}
Let $x \in St$. We say that $NS \restriction x$ is $\gamma$-saturated iff for every stationary almost disjoint family $\mathcal{A} \subseteq \mathfrak{P}(x)$ we have $\vert \mathcal{A}\vert < \gamma$.
\end{definition}

Obviously, this definition agrees with the usual definition of saturation (see \cite{Jech}).\\
\\
The next lemma will summarize some properties of $\text{Spec}_{\text{sad}}$:

\begin{lemma}
The following holds true for $\kappa$ regular uncountable:
\begin{enumerate}
\item By Lemma \ref{3L4} we have $\kappa \in \text{Spec}_{\text{sad}}$.
\item By \cite{ShGitik} we have $NS$ is not $\kappa^+$-saturated for $\kappa \geq \omega_2$, hence $\{\kappa\} \subsetneq \text{Spec}_{\text{sad}}$.
\item If $\diamondsuit_\kappa(\kappa)$ holds (see Definition \ref{3D1}), then $2^\kappa \in \text{Spec}_{\text{sad}}$.
\item By \cite{Gitik} it is consistent that $\kappa$ is inaccessible and there exists $x \in St$ such that $x \cap \{i< \kappa \colon \text{cf}(i) =j\} \in St$ for all cardinals $j< \kappa$ and $NS \restriction x$ is $\kappa^+$-saturated.\\
By \cite{JechWoodin} it is consistent that $\kappa$ is Mahlo and $NS \restriction \text{Reg}$ is $\kappa^+$-saturated.
\end{enumerate}
\end{lemma}

\begin{question}
Is it consistent that $NS$ is $2^\kappa$-saturated for $\kappa$ inaccessible? Is it even consistent that $NS$ is $\kappa^{++}$-saturated and $2^\kappa$ is very large?\\
In \cite{ShGitik} the authors ask whether the following is consistent for $\kappa$ inaccessible: $\forall x \in St \, \exists y \in St \colon y \subseteq x \land NS \restriction y$ is $\kappa^+$-saturated.
\end{question}

Also the stationary pseudo intersection number $\mathfrak{p}_\kappa^{cl}$ and the stationary tower number $\mathfrak{t}_\kappa^{cl}$ are trivial:

\begin{lemma}
Let $\kappa$ be regular uncountable. Then $\mathfrak{p}_\kappa^{cl}= \mathfrak{t}_\kappa^{cl}= \kappa$.
\end{lemma}

\begin{proof}
It will suffice to show that there exists a decreasing sequence $(x_i)_{i< \kappa}$ of stationary sets such that $\triangle_{i< \kappa} \, x_i= \{0\}$: Assume that $x^*$ is a stationary pseudo intersection of $(x_i)_{i<\kappa}$. Again define $f \colon x^* \rightarrow \kappa$ such that $f(j):= \min\{i < \kappa \colon j \notin x_i\}$ and again we note that $\{j \in x^* \colon f(j) < j\} \in NS$. Hence, $x^* \subseteq_{cl}^* \triangle_{i< \kappa}\, x_i$ must hold, which leads to a contradiction.\\
Therefore, let us show that there exists such a sequence $(x_i)_{i<\kappa}$. Let $E_\omega^\kappa:=\{i< \kappa \colon \text{cf}(i)=\omega\}$ and for every $k \in E_\omega^\kappa$ let $(j_n^k)_{n< \omega}$ be a cofinal sequence in $k$. We claim that there exists $n^* < \omega$ such that for every $i< \kappa$ the set $x_i:=\{k < \kappa \colon j_{n^*}^k \geq i\}$ is stationary. Assume towards a contradiction that for every $n< \omega$ there exist $i_n < \kappa$ such that $x_{i_n} \in NS$ and let $cl_n$ be a club disjoint to $x_{i_n}$. We define $i^*:= \sup_{n< \omega} \, i_n$ and $cl^*:=\bigcap_{n<\omega} cl_n$. Let $k^* \in E_\omega^\kappa \cap cl^*$ with $k^* > i^*$. Then it follows that $j_n^{k^*} < i^*$ for every $n< \omega$. But this contradicts the assumption that $(j_n^{k^*})_{n< \omega}$ is cofinal in $k^*$.\\
Hence let $n^*$ and $(x_i)_{i<\kappa}$ be as defined above. It remains to be shown that $\triangle_{i<\kappa} \, x_i=\{0\}$. Assume towards a contradiction that there exists $k >0$ such that $k \in \triangle_{i<\kappa} \, x_i$. This means that $j_{n^*}^k \geq i$ for every $i< k$. But this is a contradiction.
\end{proof}

Next, we investigate the stationary reaping number $\mathfrak{r}_\kappa^{cl}$:

\begin{theorem} \label{3T4}
$\mathfrak{r}_\kappa^{cl} \geq \kappa$ for $\kappa$ inaccessible.
\end{theorem}

\begin{proof}
Let $(x_i)_{i< \lambda}$ with $\lambda < \kappa$ be a family of stationary sets and w.l.o.g. assume that $\kappa \subseteq_{cl}^* \bigcup_{i< \lambda} x_i$. Assume that $(x_{i,j})_{j< \lambda}$ is a partition of $x_i$ into $\lambda$ many stationary sets and define $x_{i, \lambda}:=\kappa\setminus x_i$ for every $i< \lambda$. We will find a common refinement of the partitions $(x_{i,j})_{j \leq \lambda}$.\\
For every $s \in (\lambda+1)^\lambda$ define $y_s:= \bigcap_{i< \lambda} x_{i, s(i)}$. Clearly, if $s_1, s_2 \in (\lambda+1)^\lambda$ with $s_1 \neq s_2$ then $y_{s_1} \cap y_{s_2} = \emptyset$. Now set $S:=\{s \in (\lambda+1)^\lambda \colon y_s \in St\}$ and note that since $(\lambda+1)^\lambda < \kappa$ and every $x_{i,j}= \bigcup_{s \in (\lambda+1)^\lambda,\, s(i)=j} \, y_s$, we clearly have that $\kappa \subseteq_{cl}^* \bigcup_{s \in S} y_s$ and $(y_s)_{s \in S}$ refines every partition $(x_{i,j})_{j \leq \lambda}$.\\
Since the $y_s$ are pairwise disjoint, one can now easily construct a set $y^* \in St$ which stationarily splits $y_s$ for every $s \in S$, and hence stationarily splits $x_i$ for every $i< \lambda$.
\end{proof}

We will later see that $\mathfrak{r}_\kappa^{cl} > \kappa$ can be forced.

\begin{definition} \label{3D1}
Let $x \subseteq \kappa$ be stationary. We say that $\diamondsuit_\kappa(x)$ holds iff there exists a sequence $(s_i)_{i \in x}$ with $s_i \subseteq i$ and for every $y \subseteq \kappa$ the set $\{i \in x \colon y \restriction i = s_i\}$ is stationary (see \cite{Jech}).
\end{definition}

\begin{question}
Is $\mathfrak{r}_\kappa^{cl} = \kappa$ consistent? Does $\forall x \in St \colon \diamondsuit_\kappa(x)$ imply $\mathfrak{r}_\kappa^{cl} > \kappa$? How does $\mathfrak{r}_\kappa^{cl}$ relate to $\mathfrak{r}_\kappa$?
\end{question}

Concerning the various definitions of ultrafilter numbers:

\begin{lemma}
For $\kappa$ measurable we have:
\begin{enumerate}
\item $\kappa^+ \leq \mathfrak{r}_\kappa \leq \mathfrak{u}_\kappa \leq \mathfrak{u}_\kappa^{cl} \leq  \mathfrak{u}_\kappa^{nm}$
\item $\kappa \leq \mathfrak{r}_\kappa^{cl} \leq \mathfrak{u}_\kappa^{cl^*}\leq \mathfrak{u}_\kappa^{nm^*}$, $\mathfrak{u}_\kappa^{cl^*}\leq \mathfrak{u}_\kappa^{cl}$ and $\mathfrak{u}_\kappa^{nm^*}\leq \mathfrak{u}_\kappa^{nm}$
\item $\mathfrak{u}_\kappa^{me}=\mathfrak{u}_\kappa^{nm}$ and $\kappa^+ \leq \mathfrak{u}_\kappa^{nm^*}$
\end{enumerate}
\end{lemma}

\begin{proof}
1.) and 2.) should be obvious (using Theorem \ref{3T4}). Hence let us prove 3.): We clearly have $\mathfrak{u}_\kappa^{me} \leq \mathfrak{u}_\kappa^{nm}$. On the other hand let $\mathcal{U}$ be a measure such that there exists a base $\mathcal{B}$ of $\mathcal{U}$ with $\vert \mathcal{B} \vert = \mathfrak{u}_\kappa^{me}$. Let $V^\kappa / \mathcal{U}$ denote the ultrapower of $V$ modulo $\mathcal{U}$, let $M:=\text{mos}(V^\kappa / \mathcal{U})$ be the transitive collapse and $j \colon V \rightarrow M$ the elementary embedding. Pick $f \colon \kappa \rightarrow \kappa$ such that $\kappa= \text{mos}([f]_{\mathcal{U}})$. Then $\mathcal{V}:=\{x \subseteq \kappa \colon \kappa \in j(x)\}$ is a normal measure and it easily follows that $\mathcal{V}=\{x \subseteq \kappa \colon \exists y \in \mathcal{U}\,\, f[y] \subseteq x\}$.
Hence $f[\mathcal{B}]$ is a base of $\mathcal{V}$ and  $ \mathfrak{u}_\kappa^{nm}\leq \mathfrak{u}_\kappa^{me}$ follows.\\
To show that $\kappa^+ \leq \mathfrak{u}_\kappa^{nm^*}$ we assume towards a contradiction that $\mathcal{U}$ is a normal ultrafilter and there exists $\mathcal{B} \subseteq \mathcal{U}$ with $\vert \mathcal{B} \vert= \kappa$ such that $\{ y \in St \colon \exists x \in \mathcal{B}\, \exists cl \in Cl \,\, y= x \cap cl\}$ is a base of $\mathcal{U}$. If we enumerate $\mathcal{B}$ as $(x_i)_{i< \kappa}$ then we see that $\triangle_{i<\kappa} \, x_i \in \mathcal{U}$. But for all $x \in \mathcal{B}$ we have $x \nsubseteq_{cl}^* \triangle_{i<\kappa} \, x_i$ which leads to a contradiction.
\end{proof}

\begin{lemma}
By \cite{Montoya} the following is consistent: $\kappa^+ < \mathfrak{r}_\kappa = \mathfrak{u}_\kappa^{nm} < 2^\kappa$.
\end{lemma}

\begin{question}
Are there any other provable relations between the various ultrafilter numbers? Are $\mathfrak{u}_\kappa^{cl^*}<\mathfrak{u}_\kappa^{cl}$ or $\mathfrak{u}_\kappa^{nm^*}<\mathfrak{u}_\kappa^{nm}$ consistent? Is even $\mathfrak{u}_\kappa^{cl^*}=\kappa$ consistent?
\end{question}

Let us now investigate the stationary splitting number $\mathfrak{s}_\kappa^{cl}$:

\begin{theorem} \label{3T1}
For $\kappa$ regular uncountable we have $\mathfrak{s}_\kappa^{cl} \geq \kappa$ iff $\kappa$ is inaccessible.
\end{theorem}

\begin{proof}
We follow the proof of \cite{Suzuki}. First assume that $\kappa$ is not inaccessible, hence there exists a minimal  $\lambda < \kappa$ such that $2^\lambda\geq \kappa$. Let $f \colon \kappa \rightarrow 2^\lambda$ be injective and for every  $s \in 2^{<\lambda}$ define $x_s:=\{i< \kappa \colon s \triangleleft f(i)\}$. We set $X:=\{x_s \colon s \in 2^{<\lambda} \land x_s \in St\}$ which is of size $2^{<\lambda} < \kappa$, and claim that $X$ is a stationary splitting family. Towards a contradiction assume that $y \in St$ is not stationarily split by $X$. It follows that the set $S:=\{ s \in 2^{<\lambda} \colon y \subseteq_{cl}^* x_s \}$ is linearly ordered by $\triangleleft$ , because for incompatible $s_1, s_2 \in 2^{<\lambda}$ we have that $x_{s_1}$ and $x_{s_2}$ are disjoint. Let us define $t:= \bigcup_{s \in S} s \in 2^\lambda$. Now we can deduce that $y \subseteq f^{-1}(\{t\}) \cup \bigcup_{s \in 2^{<\lambda} \setminus S}\, x_s$. However, this leads to a contradiction, because $y$ would be covered by a union of ${<}\kappa$ many non-stationary sets.\\
On the other hand assume that $\kappa$ is inaccessible and let $X \subseteq St$ be of size $\lambda<\kappa$. Let $\theta$ be large enough and choose an elementary submodel $M \prec H(\theta)$ with $\kappa, X \in M$, $X, 2^\lambda \subseteq M$ and $\vert M \vert < \kappa$. Now pick $i^* > \sup(M \cap \kappa)$ such that $i^* \in \bigcap_{cl \in Cl \cap M} cl$. The ordinal $i^*$ induces a partition $Y_0, Y_1$ of $X$: set $Y_0:=\{x \in X \colon i^* \notin x\}$ and $Y_1:=\{x \in X \colon i^* \in x\}$. Since $2^\lambda \subseteq M$ we can deduce that also $Y_0, Y_1 \in M$, and hence $y:= \bigcap Y_1 \setminus \bigcup Y_0 \in M$. If we can show that $y \in St$, this will imply that $X$ is not a stationary splitting family. To this end let $cl \in Cl \cap M$ be arbitrary, and we obviously have $H(\theta)\vDash i^* \in y \cap cl$. By elementarity it follows that $M \vDash y \cap cl \neq \emptyset$, and since $cl$ was arbitrary, we can deduce that $M \vDash y \in St$. Again by elementarity we have $ y\in St$.
\end{proof}

The following definition already appeared in \cite{Schlicht}:

\begin{definition}
Let $F \subseteq \mathfrak{P}(\kappa)$ be a uniform filter \footnote{In particular we can assume that $F$ contains the co-bounded filter.}, i.e. for every $x \in F$ we have $\vert x \vert = \kappa$. We define:
\begin{itemize}
\item $F$ is ${<}\kappa$-complete$^*$ iff for every $\lambda < \kappa$ and every $(x_i)_{i< \lambda}$ with $x_i \in F$ we have $\vert \bigcap_{i < \lambda} x_i \vert = \kappa$. \footnote{Note that any ${<}\kappa$-complete$^*$ filter $F$ can be extended to a ${<}\kappa$-complete filter $\tilde{F}$.}
\item $F$ is normal$^*$ iff for every $(x_i)_{i< \kappa}$ with $x_i \in F$ we have that $\triangle_{i< \kappa} \, x_i$ is stationary.
\item $F$ measures a set $X \subseteq \mathfrak{P}(\kappa)$ iff for every $x \in X$ either $x \in F $ or $\kappa \setminus x \in F$ holds true.
\end{itemize}
\end{definition}

Note that we explicitly do not require that the (diagonal) intersection is again an element of $F$.  Clearly, if $F$ is normal$^*$, then it is also ${<}\kappa$-complete$^*$.

\begin{definition}
We say that $\kappa$ has the normal$^*$ filter property iff for every $X \subseteq \mathfrak{P}(\kappa)$ of size ${\leq}\kappa$ there exists a normal$^*$ filter $F$ measuring $X$.
\end{definition}

The following notion clearly strengthens weak compactness and is downwards absolute to $L$ (see \cite{Jensen}):

\begin{definition}
Recall that $\kappa$ is ineffable iff for every partition $f \colon [\kappa]^2 \rightarrow \{0,1\}$ there exists a stationary homogeneous set $x \subseteq \kappa$.
\end{definition}

The following theorem was proven in \cite{Zwicker}:

\begin{theorem} \label{3T2}
Let $\kappa$ be regular uncountable. Then $\kappa$ has the normal$^*$ filter property iff $\kappa$ is ineffable.
\end{theorem}

\begin{theorem} 
For $\kappa$ regular uncountable we have $\mathfrak{s}_\kappa^{cl} > \kappa$ iff $\kappa$ is ineffable.
\end{theorem}

\begin{proof}
We will show that $\mathfrak{s}_\kappa^{cl} > \kappa$ iff $\kappa$ has the normal$^*$ filter property. Then this theorem follows by Theorem \ref{3T2}.\\
Let us first assume that $\mathfrak{s}_\kappa^{cl} > \kappa$ and let $X \subseteq \mathfrak{P}(\kappa)$ be of size ${\leq}\kappa$. We will show that there exists a normal$^*$ filter $F$ measuring $X$. W.l.o.g. $X$ is closed under compliments.
 Since $\mathfrak{s}_\kappa^{cl} > \kappa$ there exists $y^* \in St$ such that $X$ does not stationarily split $y^*$. Now we define $F:=\{ x \in X \colon y^* \subseteq_{cl}^* x\}$ and note that $F$ is obviously an ultrafilter on $X$. We claim that $F$ is normal$^*$. Let $(x_i)_{i< \kappa}$  with $x_i \in F$ be arbitrary and $cl_i \in Cl$ with $y^* \cap cl_i \subseteq x_i$. Then $\triangle_{i< \kappa} \, x_i \supseteq \triangle_{i<\kappa} \, y^* \cap cl_i= y^* \cap \triangle_{i< \kappa} \, cl_i$ which is clearly stationary.\\
On the other hand assume that $\kappa$ has the normal$^*$ filter property and let $X\subseteq St$ be of size $\kappa$. Then there exists a normal$^*$ filter $F$ measuring $X$, and enumerate $X$ as $(x_i)_{i< \kappa}$. Define $y_i:= x_i$ if $x_i \in F$ and $y_i:=\kappa \setminus x_i$ else. Since $F$ is normal$^*$, we can deduce that $y^*:=\triangle_{i< \kappa} \, y_i \in St$. But no $x_i \in X$ can stationarily split $y^*$, hence  $\mathfrak{s}_\kappa^{cl} > \kappa$.
\end{proof}


Before we can state the next theorem, we need the following definition:

\begin{definition}
Let $\alpha$ be a measurable cardinal and let $\mathcal{U}_0$, $\mathcal{U}_1$ and $\mathcal{U}$ be normal measures on $\alpha$. We recall (see \cite{Jech}):
\begin{itemize}
\item the Mitchell order: $\mathcal{U}_0 \triangleleft \mathcal{U}_1$ iff $\mathcal{U}_0 \in V^\kappa /\mathcal{U}_1$, i.e. $\mathcal{U}_0$ is contained in the ultrapower of $V$ modulo $\mathcal{U}_1$
\item $o(\mathcal{U}):=\sup\{o(\mathcal{U}')+1 \colon \mathcal{U}' \triangleleft \mathcal{U}\}$ the order of $\mathcal{U}$
\item $o(\alpha):=\sup\{o(\mathcal{U}')\colon \mathcal{U}' \,\, \text{is normal measure on} \,\, \alpha\}$ the order of $\alpha$
\end{itemize}
\end{definition}

It was proven by Zapletal (see \cite{Zapletal}) that $\mathfrak{s}_\kappa > \kappa^+$ has large consistency strength, and indeed the same proof shows:

\begin{theorem}
Let $\mathfrak{s}_\kappa^{cl} > \kappa^{++}$. Then there exists an inner model with a measurable cardinal $\alpha$ of order $\alpha^{++}$. \footnote{This is equivalent to $\exists \mathcal{F} \colon L[\mathcal{F}] \vDash \exists \alpha \colon \alpha \,\,\text{is measurable with order} \,\, \alpha^{++}$ (see \cite{Mitchell1}) .}
\end{theorem}

Let us now show some consistency results regarding $\mathfrak{s}_\kappa^{cl}$, $\mathfrak{b}_\kappa$, $\mathfrak{d}_\kappa$ and $\mathfrak{r}_\kappa^{cl}$. First we state a helpful tool:

\begin{lemma} \label{3L2}
Let $V \vDash x \in St$ and let $\mathcal{P}$ be a ${<}\kappa$-closed forcing. Then $V^\mathcal{P} \vDash x \in St$.
\end{lemma}

\begin{proof}
Since being stationary is a $\Pi_1^1$ statement, the lemma follows by $\Pi_1^1$-absoluteness for ${<}\kappa$-closed forcing extensions.
\end{proof}

\begin{definition}
Let $\mathcal{U}$ be a ${<}\kappa$-complete, normal ultrafilter on $\kappa$. We define $\mathbb{M}_{\mathcal{U}}$, the generalized Mathias forcing with respect to $\mathcal{U}$, as follows:
\begin{itemize}
\item A condition $p$ is of the form $(s^p,A^p)$ where $s^p \in [\kappa]^{<\kappa}$, $A^p \in \mathcal{U}$ and $\sup s^p \leq \min A^p$.
\item Let $p=(s^p,A^p)$ and $q=(t^q, B^q)$ be in $\mathbb{M}_\mathcal{U}$. We define $q \leq_{\mathbb{M}_\mathcal{U}} p$, in words $q$ is stronger than $p$, if $s^p \subseteq t^q$, $B^q \subseteq A^p$ and $t^q \setminus s^p \subseteq A^p$.
\end{itemize}
If $G$ is a $(V, \mathbb{M}_\mathcal{U})$-generic filter, we define $m_G:= \bigcup_{p \in G} s^p$.
\end{definition}

The next lemma follows immediately.

\begin{lemma}
Let $\mathcal{U}$ be a ${<}\kappa$-complete, normal ultrafilter. Then the forcing $\mathbb{M}_\mathcal{U}$ has the following properties:
\begin{itemize}
\item $\mathbb{M}_\mathcal{U}$ is $\kappa\text{-centered}_{{<}\kappa}$. In particular it is $\kappa^+$-c.c.
\item $\mathbb{M}_\mathcal{U}$ is ${<}\kappa$-directed closed.
\end{itemize}
\end{lemma}

\begin{lemma} \label{3L1}
Let $\mathcal{U}$ be a ${<}\kappa$-complete, normal ultrafilter on $\kappa$ and let $V \vDash x \in St$. Then $ \Vdash_{\mathbb{M}_\mathcal{U}} \dot{m_G} \in St \land (\dot{m_G} \subseteq_{cl}^* x  \vee \dot{m_G} \cap x \in NS)$.
\end{lemma}

\begin{proof}
If $x \in \mathcal{U}$ then clearly $\Vdash_{\mathbb{M}_\mathcal{U}} \dot{m}_G \subseteq^* x$. On the other hand, if $x \notin \mathcal{U}$ then $\Vdash_{\mathbb{M}_\mathcal{U}} \dot{m}_G \cap x \,\, \text{is bounded}$. Hence, it remains to be shown that $\Vdash_{\mathbb{M}_\mathcal{U}} \dot{m}_G \in St$. To this end let $ p \in \mathbb{M}_\mathcal{U}$ and $\dot{cl}$ be a $\mathbb{M}_\mathcal{U}$-name for a club. Let $(p_i)_{i < \kappa}$ be a  decreasing sequence of conditions below $p$ interpreting $\dot{cl}$ as $cl^* \in V$, and w.l.o.g assume that $p_\lambda= \inf_{i< \lambda} p_i$ for every limit $\lambda < \kappa$. Let $A^*:=\triangle_{i< \kappa}\, A^{p_i}$ denote the diagonal intersection of the $A^{p_i}$, and since $\mathcal{U}$ is a normal measure, we have that $A^* \in \mathcal{U}$. Hence, $A^* \cap \text{Lim}(cl^*) \neq \emptyset$ where $\text{Lim}(cl^*)$ is the club consisting only of the limit points of $cl^*$, and pick $i^* \in A^* \cap cl^*$. It follows that $i^* \in A^{p_{i^*}}$ and $p_{i^*} \Vdash_{\mathbb{M}_\mathcal{U}} i^* \in \dot{cl}$. If we define a condition $q:=(s^{p_{i^*}} \cup \{i^*\}, A^{p_{i^*}} \setminus \{i^*\})$ then trivially $q \leq_{\mathbb{M}_\mathcal{U}} p_{i^*}$ and $q \Vdash_{\mathbb{M}_\mathcal{U}} i^* \in \dot{m}_G \cap \dot{cl}$. Hence, $\Vdash_{\mathbb{M}_\mathcal{U}} \dot{m}_G \in St$.
\end{proof}

\begin{theorem}
Let $\kappa$ be supercompact and indestructible by ${<}\kappa$-directed closed forcing posets (see \cite{Laver2}). Define a ${<}\kappa$-support iteration $( \mathbb{P}_\alpha, \dot{\mathbb{Q}}_\beta \colon \alpha \leq \kappa^{++}, \beta < \kappa^{++})$ such that $\Vdash_{\mathbb{P}_\alpha} \dot{\mathbb{Q}}_\alpha = \mathbb{M}_{\dot{\mathcal{U}}_\alpha}$ where $\dot{\mathcal{U}}_\alpha$ is a $\mathbb{P}_\alpha$-name for a ${<}\kappa$-complete, normal ultrafilter. Furthermore, assume that $V \vDash 2^\kappa = \kappa^+$. Then $V^\mathbb{P} \vDash \mathfrak{s}_\kappa^{cl}=\mathfrak{b}_\kappa= \mathfrak{d}_\kappa = \mathfrak{r}_\kappa^{cl}= 2^\kappa= \kappa^{++}$.
\end{theorem}

\begin{proof}
Since $\mathbb{P}$ satisfies the $\kappa^+$-c.c. and for every $\alpha < \kappa^{++}$ the forcing $\mathbb{P}_\alpha$ has a dense subset of size $\kappa^+$, we can deduce that $V^\mathbb{P}\vDash 2^\kappa= \kappa^{++}$. It is easy to see that $V^\mathbb{P}\vDash\mathfrak{b}_\kappa= \mathfrak{d}_\kappa =\kappa^{++}$. Since $\mathbb{P}$ has ${<}\kappa$-support, it follows that $\mathbb{P}$ adds $\kappa$-Cohen reals, hence $V^\mathbb{P}\vDash\mathfrak{r}_\kappa^{cl}=\kappa^{++}$ (see Lemma \ref{3L3}). Now if $V^\mathbb{P} \vDash X\subseteq St$ is a set of size $\leq \kappa^+$ in $V^\mathbb{P}$, then by the $\kappa^+$-c.c. there exists $\alpha < \kappa^{++}$ such that $X \in V^{\mathbb{P}_\alpha}$ and by downward absoluteness $V^{\mathbb{P}_\alpha} \vDash X \subseteq St$. By Lemma \ref{3L1} $X$ is not a stationary splitting family in $V^{\mathbb{P}_{\alpha+1}}$, hence by Lemma \ref{3L2} $X$ cannot be a stationary splitting family in $V^\mathbb{P}$.
\end{proof}

\begin{lemma} \label{3L3}
Let $\mathbb{C}_\kappa$ denote the $\kappa$-Cohen forcing and let $G$ be a $(V, \mathbb{C}_\kappa)$-generic filter. Let $c_G \subseteq \kappa$ denote the $\kappa$-Cohen real added by $G$, and let $V \vDash x \in St$. Then $V^{\mathbb{C}_\kappa} \vDash c_G \,\, \text{stationarily splits}\,\, x$.\\
Furthermore, if $\mathbb{P}=\prod_{\alpha < \kappa^+} \mathbb{C}_\kappa$ denotes the ${<}\kappa$-support product of $\kappa$-Cohen forcing, then $V^\mathbb{P} \vDash (c_\alpha)_{\alpha< \kappa^+}\,\, \text{is a stationary splitting family}$.
\end{lemma}

\begin{proof}
We proceed similarly to the proof of \ref{3L1}: Let $p \in \mathbb{C}_\kappa$ and $\dot{cl}$ a $\mathbb{C}_\kappa$-name for a club. Let $(p_i)_{i< \kappa}$ be a decreasing sequence below $p$ interpreting $\dot{cl}$ as $cl^* \in Cl \cap V$. Again, w.l.o.g. assume that $p_\lambda= \inf_{i< \lambda} p_i$ for every limit $\lambda < \kappa$. Since $x$ is stationary in $V$, we can find $i^* \in x \cap \text{Lim}(cl^*)$ where $\text{Lim}(cl^*)$ is again the club consisting only of the limit points of $cl^*$. Hence, there are $q_0, q_1 \in \mathbb{C}_\kappa$ below $p_{i^*}$ such that $q_0 \Vdash_{\mathbb{C}_\kappa} (x \setminus c_{\dot{G}}) \cap \dot{cl} \neq \emptyset$ and $q_1 \Vdash_{\mathbb{C}_\kappa} x \cap c_{\dot{G}} \cap \dot{cl} \neq \emptyset$.\\
Let $\dot{x}$ be a $\mathbb{P}$-name for a stationary set in $V^\mathbb{P}$. By the $\kappa^+$-c.c. of $\mathbb{P}$ it follows that there exists $\alpha < \kappa^+$ such that $\dot{x}$ is a $\mathbb{P}_\alpha$-name, where $\mathbb{P}_\alpha:= \prod_{\beta < \alpha} \mathbb{C}_\kappa$. By the above $V^{\mathbb{P}_{\alpha+1}} \vDash c_\alpha \,\, \text{stationarily splits} \,\, x$. By Lemma \ref{3L2} we have $V^\mathbb{P} \vDash c_\alpha \,\, \text{stationarily splits} \,\, x$.
\end{proof}

The following proof already appeared in a similar version in \cite{207}:

\begin{theorem} \label{3T3}
Let $\kappa$ be supercompact and indestructible by ${<}\kappa$-directed closed forcing posets. Let $V\vDash 2^\kappa = \kappa^+$ and define $\mathbb{R}:=\mathbb{P} \star \dot{\mathbb{Q}}$ where $\mathbb{P}=\prod_{\alpha < \kappa^+} \mathbb{C}_\kappa$ and $\dot{\mathbb{Q}}$ is a $\mathbb{P}$-name for a $\kappa^{++}$ iteration of $\kappa$-Hechler forcing $\mathbb{H}_\kappa$ with ${<}\kappa$-support. Then $V^\mathbb{R}\vDash \mathfrak{s}_\kappa^{cl} = \kappa^+ \land \mathfrak{b}_\kappa = \mathfrak{d}_\kappa=\mathfrak{r}_\kappa^{cl}= 2^\kappa= \kappa^{++}$.
\end{theorem}

\begin{proof}
Obviously, $\mathfrak{b}_\kappa=\mathfrak{d}_\kappa =\kappa^{++}$. Since $\mathbb{H}_\kappa$ adds $\kappa$-Cohen reals, we can deduce by \ref{3L3} that $\mathfrak{r}_\kappa^{cl}= \kappa^{++}$. Since $\kappa$ remains ineffable in $V^\mathbb{R}$ it follows that $\mathfrak{s}_\kappa^{cl} \geq \kappa^+$. It remains to be shown that $\mathfrak{s}_\kappa^{cl} \leq \kappa^+$. To this end we will show that $(c_\alpha)_{\alpha< \kappa^+}$ remains a splitting family in $V^\mathbb{R}$ where the $(c_\alpha)_{\alpha<\kappa^+}$ are the generic $\kappa$-Cohen reals added by $\mathbb{P}$.\\
Towards a contradiction assume that $\dot{x}$ is a $\mathbb{R}$-name and $(p, \dot{q})$ a condition in $\mathbb{R}$  such that $(p,\dot{q}) \Vdash_{\mathbb{R}} \dot{x} \in St \land ( \forall \alpha< \kappa^+ \colon \dot{x} \subseteq_{cl}^* \dot{c}_\alpha \vee \dot{x} \cap \dot{c}_\alpha \in NS )$. Since $\mathbb{R}$ satisfies the $\kappa^+$-c.c.  we can find $\alpha^* < \kappa^+$ such that the $\mathbb{R}$-name $\dot{x}$ does not depend on $\dot{c}_{\alpha^*}$. Since $\mathbb{P}$ is ${<}\kappa$-closed and $\Vdash_\mathbb{P} \dot{\mathbb{Q}} \,\, \text{has}\,\, {<}\kappa \text{-support}$ and is ${<}\kappa$-closed, we obviously have $$\Vdash_\mathbb{P} \{ q \in \dot{\mathbb{Q}} \colon \text{dom}(q) \in \check{V} \land \exists \bar{\rho} \in (\kappa^{<\kappa})^{\text{dom}(q)} \cap \check{V}$$ $$ \forall \alpha \in \text{dom}(q)  \,\, \exists \dot{f}\,\, \Vdash_{\dot{\mathbb{Q}}} \dot{q}(\alpha)=(\bar{\rho}(\alpha), \dot{f})\} \, \text{is dense in}\,\, \dot{\mathbb{Q}}$$
Hence, we can pick a condition $(p',\dot{q'}) \leq_\mathbb{R} (p, \dot{q})$ such that all trunks of $(p', \dot{q'})$ are ground model objects, and $(p',\dot{q'})$ decides whether $\dot{x} \subseteq_{cl}^* c_{\alpha^*}$ or $\dot{x} \cap c_{\alpha^*} \in NS$, w.l.o.g. assume that $(p',\dot{q'}) \Vdash_\mathbb{R} \dot{x} \subseteq_{cl}^* c_{\alpha^*}$. Now we define an automorphism $\pi$ of $\mathbb{P}$ which fixes $\prod_{\alpha \in \kappa\setminus \{\alpha^*\}} \mathbb{C}_\kappa$ and $\Vdash_\mathbb{P} \dot{c}_{\alpha^*} \cap \pi(\dot{c}_{\alpha^*}) \subseteq \text{dom}(p'(\alpha^*))$, in particular $p' = \pi(p')$. Now $\pi$ induces an automorphism $\tilde{\pi}$ of $\mathbb{R}$, and since all trunks of $(p', \dot{q'})$ are ground model objects, we can deduce that $p' \Vdash_\mathbb{P} \dot{q'} \,\, \text{and} \,\, \tilde{\pi}(\dot{q'}) \,\, \text{are compatible in} \,\, \dot{\mathbb{Q}}$. Hence there exists a condition $(p', \dot{r}) \leq_\mathbb{R} (p', \dot{q'}), (p', \tilde{\pi}(\dot{q'}))$, and since $\Vdash_\mathbb{R} \dot{x} = \tilde{\pi}(\dot{x})$ we can deduce that $(p', \dot{r}) \Vdash_\mathbb{R} \dot{x} \subseteq_{cl}^* c_{\alpha^*} \land \dot{x} \subseteq_{cl}^* \tilde{\pi}( c_{\alpha^*})$. But this immediately leads to a contradiction.
\end{proof}

\begin{lemma}
Let $\kappa$ be supercompact and indestructible by ${<}\kappa$-directed closed forcing posets. Let $V\vDash 2^\kappa = \kappa^+$ and define $\mathbb{P}=\prod_{\alpha < \kappa^{++}} \mathbb{C}_\kappa$. Then $V^\mathbb{P}\vDash \mathfrak{s}_\kappa^{cl}= \mathfrak{b}_\kappa = \kappa^+ \land \mathfrak{d}_\kappa=\mathfrak{r}_\kappa^{cl}= 2^\kappa= \kappa^{++}$.
\end{lemma} 

\begin{proof}
The lemma immediately follows from the proof of Theorem \ref{3T3}.
\end{proof}

\begin{lemma}
Let $\kappa$ be supercompact and indestructible by ${<}\kappa$ directed-closed forcing posets. Let $V\vDash 2^\kappa = \kappa^+$ and define $\mathbb{P}=\prod_{\alpha < \kappa^{++}} \mathbb{S}_\kappa$, i.e. a $\kappa^{++}$ product of $\kappa$-Sacks forcing with ${\leq}\kappa$-support. Then $V^\mathbb{P}\vDash  \mathfrak{b}_\kappa =\mathfrak{d}_\kappa= \kappa^+ \land 
\mathfrak{r}_\kappa^{cl}= 2^\kappa= \kappa^{++}$.
\end{lemma}

\begin{proof}
Since $\mathbb{P}$ is $\kappa^\kappa$-bounding, we have $V^\mathbb{P} \vDash \mathfrak{b}_\kappa = \mathfrak{d}_\kappa=\kappa^+$. It also follows that $V^\mathbb{P} \vDash Cl \cap V \,\, \text{is cofinal in}\,\, Cl$, and therefore it is easy to see that $V^\mathbb{P} \vDash \forall \alpha < \kappa^{++} \colon  s_\alpha$ stationarily splits $St \cap V^{\mathbb{P}_\alpha}$, where $\mathbb{P}_\alpha:= \prod_{\beta < \alpha} \mathbb{S}_\kappa$. Hence, $V^\mathbb{P} \vDash \mathfrak{r}_\kappa^{cl}=2^\kappa=\kappa^{++}$.
\end{proof}

It seems very reasonable to conjecture that $V^\mathbb{P}\vDash \mathfrak{s}_\kappa^{cl} = \kappa^+$.

\begin{question}
Is $\mathfrak{b}_\kappa < \mathfrak{s}_\kappa^{cl}$ consistent?
Is even $\mathfrak{d}_\kappa < \mathfrak{s}_\kappa^{cl}$ consistent? How does $\mathfrak{s}_\kappa^{cl}$ relate to $\mathfrak{s}_\kappa$?
\end{question}

\newpage

\bibliography{refs} 
\bibliographystyle{alpha}

\end{document}